\documentclass[12pt]{amsart}
\usepackage{amsmath}
\usepackage{amssymb,amsthm}
\usepackage{amscd}
\usepackage{xypic}

\newtheorem{theorem}{Theorem}[section]

\newtheorem{corollary}[theorem]{Corollary}
\newtheorem{proposition}[theorem]{Proposition}

\theoremstyle{definition}
\newtheorem{definition}[theorem]{Definition}

\theoremstyle{remark}

\numberwithin{equation}{section}

\begin{document}
\title{The Invariance and the General CCT Theorems\\ }
\subjclass[2000]{Primary: 18, Secondary: 16. }
\author[Alin  Stancu]{\textbf{Alin  Stancu}\\
Department of Mathematics, Columbus State University,\\ Columbus, GA
31907, USA\\}

\email{stancu\_alin1@colstate.edu}

\date{}

\begin{abstract}

The \begin{it} Invariance Theorem \end{it} of M. Gerstenhaber and S. D. Schack states that if $\mathbb{A}$ is a diagram of algebras then the subdivision functor induces a natural isomorphism between the Yoneda cohomologies of the category $\mathbb{A}$-$\mathbf{mod}$ and its subdivided category $\mathbb{A}'$-$\mathbf{mod}$. In this paper we generalize this result and show that the subdivision functor is a full and faithful functor between two suitable derived categories of  $\mathbb{A}$-$\mathbf{mod}$ and $\mathbb{A}'$-$\mathbf{mod}$. This result combined with our work in [5] and [6], on the $Special$ $Cohomology$ $Comparison$ $Theorem$, constitutes a generalization of M. Gerstenhaber and S. D. Schack's $General$ $Cohomology$ $Comparison$ $Theorem$ ($\mathbf{CCT}$).

\end{abstract}

\maketitle

\newpage


\section{Preliminaries}

Let $k$ be a commutative ring and $\mathcal{C}$ be a small category, $i.e$ a category whose class of objects is a set. The objects of $\mathcal{C}$
will be denoted $h, i, j,$ etc. and the maps will be $u, v, w,$ etc. We will write $\mathrm{Hom}_{\mathcal{C}}(i, j)$ for the set of maps $i\rightarrow j$ and denote the domain and the codomain of a map $v$ by $dv$ and $cv$, respectively. A $\mathbf{delta}$ is a
small category in which the only endomorphisms are identity maps and the condition $\mathrm{Hom}_{\mathcal{C}}(i, j)\neq 0$ implies
$\mathrm{Hom}_{\mathcal{C}}(j, i)=0$, for all $i\neq j$ in $\mathcal{C}$.

When $A$ is a $k$-algebra and
$M$ any $A$ bimodule we assume $M$ to be symmetric over $k$. ($i.e.$
$ax=xa$ for all $x\in M$ and $a\in k$.) The category of associative unital $k$-algebras will be denoted by
$k$-${\mathbf{alg}}$.
A $\mathbf{presheaf}$ of $k$-algebras
over $\mathcal{C}$ is a contravariant functor
$\mathbb{A}:\mathcal{C}^{op}\longrightarrow k$-${\mathbf{alg}}$. We
will denote $\mathbb{A}(i)$ by $\mathbb{A}^i$ and write $\varphi^v=\varphi_{\mathbb{A}}^v:\mathbb{A}^{cv}\rightarrow\mathbb{A}^{dv}$ for $\mathbb{A}(v)$.
We will adopt the terminology of [2] and will call  $\mathbb{A}$ a $\mathbf{diagram}$ over $\mathcal{C}$.

Let $\mathbb{A}$ be a diagram over $\mathcal{C}$ and $v\in\mathcal{C}$ be a map. The map $
\varphi_{\mathbb{A}}^v:\mathbb{A}^{cv}\rightarrow\mathbb{A}^{dv}$ makes any $\mathbb{A}^{dv}$-module $M$ an $\mathbb{A}^{cv}$ module.
The resulting module will be denoted by $|M|_v$. A $\mathbf{left}$ $\mathbb{A}$-$\mathbf{module}$ $\mathbb{M}$ is a presheaf (over $\mathcal{C}$)
of abelian groups such that:

1. $\mathbb{M}^i$ is an $\mathbb{A}^i$-module.

2. If $v\in\mathcal{C}$ then $T^v=T_{\mathbb{M}}^v:\mathbb{M}^{cv}\rightarrow|\mathbb{M}^{dv}|_v$ is an $\mathbb{A}^{cv}$-module map.

An $\mathbb{A}$-module map
$\eta:\mathbb{M}\longrightarrow\mathbb{N}$ is a natural
transformation in which $\eta^{i}$ is an $\mathbb{A}^{i}$-module
map $(\forall) i\in\mathcal{C}$. The category of left $\mathbb{A}$-modules will be denoted by $\mathbb{A}$-$\mathbf{mod}$. It is abelian, complete, and cocomplete.
(All constructions are made ``objectwise".) The definitions of right modules and bimodules  are similar and with the $\mathbf{opposite}$ $\mathbf{diagram}$ $\mathbb{A}^{op}$ and the $\mathbf{enveloping}$ $\mathbf{diagram}$ $\mathbb{A}^{e}=\mathbb{A}\otimes_k\mathbb{A}^{op}$ defined in the obvious way we have that the category of right $\mathbb{A}$-modules, $\mathbf{mod}$-$\mathbb{A}$ and that of $\mathbb{A}$-bimodules, $\mathbb{A}$-$\mathbf{bimod}$, are isomorphic to $\mathbb{A}^{op}$-$\mathbf{mod}$ and $\mathbb{A}^e$-$\mathbf{mod}$.

If $f:\mathcal{D}\rightarrow\mathcal{C}$ is a covariant functor between small categories then every diagram $\mathbb{A}:\mathcal{C}^{op}\rightarrow k$-$\mathbf{alg}$ over $\mathcal{C}$ defines a diagram $f^{*}\mathbb{A}:\mathcal{D}^{op}\rightarrow k$-$\mathbf{alg}$ over $\mathcal{D}$ by setting
$(f^{*}\mathbb{A})^{\sigma}=\mathbb{A}^{f\sigma}$ and $\varphi_{f^{*}\mathbb{A}}^v=\varphi_\mathbb{A}^{fv}$.
Moreover, the functor $f$ induces a functor $f^{*}:\mathbb{A}$-$\mathbf{mod}$ $\rightarrow (f^{*}\mathbb{A})$-$\mathbf{mod}$ by setting $(f^{*}\mathbb{N})^\sigma=\mathbb{N}^{f\sigma}$,
$T_{f^{*}\mathbb{N}}^v=T_{\mathbb{N}}^{fv}$, and $(f^*\eta)^\sigma=\eta^{f\sigma}$. Note that $f^*$ is an exact embedding.

In [3] M. Gerstenhaber and S. D. Schack proved that the functor $f^*$  has both a left and a right adjoint. Because we will use the left adjoint to prove a generalization of the $Subdivision$ $Theorem$ we include  M. Gerstenhaber and S. D. Schack's description.

Let $f:\mathcal{D}\rightarrow\mathcal{C}$ be a functor as above and $i$ an object in $\mathcal{C}$. Then the $\mathbf{comma}$ $\mathbf{category}$ $\mathbf{i/f}$ is the category whose objects are the $\mathcal{C}$-maps $\xymatrix{ i\ar[r]^{w}&f\sigma}$, where $\sigma\in\mathcal{D}$. Such an object will be denoted by $(w, \sigma)$. A map $(u,\tau)\rightarrow (w,\sigma)$ in $\mathbf{i/f}$ is simply a $\mathcal{D}$-map $\tau\rightarrow\sigma$ such that $u(fv)=w$.
Each $\mathcal{C}$ map $\xymatrix{h\ar[r]^v&i}$ induces a functor $\mathbf{i/f}\rightarrow \mathbf{h/f}$ described on objects by $(w,\sigma)\rightarrow(vw,\sigma)$.

If $v\in\mathcal{C}$ then, using the map $\varphi^v:\mathbb{A}^{cv}\rightarrow\mathbb{A}^{dv}$, we may view $\mathbb{A}^{dv}$ as a left $\mathbb{A}^{dv}$, right
$\mathbb{A}^{cv}$ module. This implies that if $M$ is any left $\mathbb{A}^{cv}$-module then $\mathbb{A}^{dv}\otimes_{\mathbb{A}^{cv}}M$ is a left $\mathbb{A}^{dv}$ module. To make the role of $v$ explicit we denote this module by $\mathbb{A}^{dv}\otimes_vM$. For $a\in\mathbb{A}^{dv}$,  $b\in\mathbb{A}^{cv}$ and $m\in M$, we have $a\otimes bm=a(\varphi^vb)\otimes m$.

The left adjoint of $f^*$ is denoted by $f_!:(f^*\mathbb{A})$-$\mathbf{mod}\rightarrow\mathbb{A}$-$\mathbf{mod}$ and defined as follows.
 Let $\mathbb{N}$ be an $(f^*\mathbb{A})$-module. For each $i\in\mathcal{C}$ and each $\xymatrix{(u,\tau)\ar[r]^v&(w,\sigma)}$ in $\mathbf{i/f}$ the map $Id\otimes T_{\mathbb{N}}^v:\mathbb{A}^i\otimes_w\mathbb{N}^{\sigma}\rightarrow\mathbb{A}^i\otimes_u\mathbb{N}^\tau$ is $\mathbb{A}^i$-linear.

The collection of all these maps defines a diagram of $\mathbb{A}^i$-modules over $\mathbf{i/f}$ by setting $$\displaystyle(f_!\mathbb{N})^i=\mathrm{colim}_{(w,\sigma)\in\mathbf{i/f}}\mathbb{A}^i\otimes_w\mathbb{N}^\sigma$$

Also, for each $v\in\mathrm{Hom}_\mathcal{C}(h,i)$ and $(w,\sigma)\in\mathbf{i/f}$ there is an $\mathbb{A}^i$-module map $\varphi_\mathbb{A}^v\otimes Id:\mathbb{A}^i\otimes_w\mathbb{N}^\sigma\rightarrow\mathbb{A}^h\otimes_{vw}\mathbb{N}^\sigma$. The universal property of colimits implies that the functor $\mathbf{i/f}\rightarrow\mathbf{h/f}:(w,\sigma)\rightarrow (vw,\sigma)$ induces an $\mathbb{A}^i$-module map $T^v:(f_!\mathbb{N})^i\rightarrow(f_!\mathbb{N})^h$
and that $T^uT^v=T^{uv}$. This implies that these modules and maps form an $\mathbb{A}$-module $f_!\mathbb{N}$. By the universality property of colimits each $f^*\mathbb{A}$-module map $\mathbb{N}\rightarrow\mathbb{M}$ induces an $\mathbb{A}$-module map $f_!\mathbb{N}\rightarrow f_!\mathbb{M}$, so $f_!$ is a functor.
For a proof that $f_!$ is a left adjoint of $f^*$ the reader could see [2].

Yoneda cohomology of the category $\mathbb{A}$-$\mathbf{mod}$ is closely related to the notion  of
``allowable'' map. These maps will also play an important role in defining the relative derived category of $\mathbb{A}$-$\mathbf{mod}$, so we
remind the reader their definition.
A map $\eta:\mathbb{M}\longrightarrow\mathbb{N}$ is called $\mathbf{allowable}$
if $(\forall) i\in\mathcal{C}$ the map
$\eta^{i}:\mathbb{M}^{i}\longrightarrow\mathbb{N}^{i}$ admits a
$k$-module splitting map
$k^{i}:\mathbb{N}^{i}\longrightarrow\mathbb{M}^{i}$ satisfying
$\eta^{i}k^{i}\eta^{i}=\eta^{i}$. We do not require the splitting
maps $k^{i}$ to be natural. An $\mathbb{A}$-module $\mathbb{P}$ is
called $\mathbf{relative}$ $\mathbf{ projective}$ if for every allowable
epimorphism $\mathbb{M}\longrightarrow\mathbb{N}$ the induced map
$\mathrm{Hom}_{\mathbb{A}}(\mathbb{P},\mathbb{M})\longrightarrow
\mathrm{Hom}_{\mathbb{A}}(\mathbb{P},\mathbb{N})$ is an
epimorphism of sets.

A $\mathbf{relative}$ $\mathbf{projective}$ $\mathbf{allowable}$
$\mathbf{resolution}$ of an $\mathbb{A}$-module $\mathbb{M}$ is an
exact sequence
$\cdots\longrightarrow\mathbb{P}_{n}\cdots\longrightarrow\mathbb{P}_{1}
\longrightarrow\mathbb{P}_{0}\longrightarrow\mathbb{M}\longrightarrow0$
in which all $\mathbb{P}_n$ are relative projective $\mathbb{A}$-modules and all maps are allowable. The category
$\mathbb{A}$-$\mathbf{mod}$ has enough relative projective modules and each
module has a relative projective allowable resolution. Moreover,
there is a functorial way of getting this type of resolutions.
The construction of such a resolution is due to M. Gerstenhaber and S.
D. Schack (see [2]). They called it the $\mathbf{gereralized}$ $\mathbf{simplicial}$ $\mathbf{bar}$ $\mathbf{(GSB)}$ $\mathbf{resolution}$.

\section{The Subdivision of a Category }

Let $\mathcal{C}$ be a small category. If $[\mathbf{p}]$ is the linearly ordered set $\{0<\dots<p\}$ viewed as a category then a $p-\mathbf{simplex}$  is a covariant functor $\sigma:[\mathbf{p}]\rightarrow\mathcal{C}$. In this case we say that the dimension of $\sigma$ is p and we write dim$\sigma$=$p$. A functor $f:[\mathbf{p}]\rightarrow [\mathbf{q}]$ is called $\mathbf{monotone}$ if and only if
$i<j$ implies $fi<fj$.

Every small category $\mathcal{C}$ has a $\mathbf{subdivision}$ $\mathcal{C}'$ which is again a category. The objects of the subdivision $\mathcal{C}'$ are the simplices of the category $\mathcal{C}$. To define the maps let $\tau$ and $\sigma$ are objects in $\mathcal{C}'$ such that $\mathrm{dim}\tau=p$ and $\mathrm{dim}\sigma=q$. A map $\tau\rightarrow\sigma$ in $\mathcal{C}'$ is a triple $[\tau, \sigma, v]$ where $\tau$ is the domain, $\sigma$ the codomain, and $v$ a map in $\mathcal{C}$ such that there exists a monotone functor $f:[\mathbf{q}]\rightarrow [\mathbf{p}]$ such that the triangle
$$\xymatrix{& [\mathbf{q}]\ar[dl]_(.5)f\ar[dr]^\sigma \\ [\mathbf{p}]\ar[rr]^\tau &&\mathcal{C}}$$ commutes and $v=\tau^{0,f0}=\tau(0\rightarrow f0)=\tau(0)\rightarrow\sigma(0)$.

Note that if $\mathrm{dim}\tau<\mathrm{dim}\sigma$ then there are no maps $\tau\rightarrow\sigma$. The composition is written in diagrammatic order
and is given by $[\tau, \sigma, u]\circ [\sigma, \omega, v]=[\tau, \omega, uv]$.
It is not hard to see that the  following proposition is true.
\begin{proposition} If $\mathcal{C}$ is a small category then $\mathcal{C}'$ is a delta. In addition, if $\mathcal{C}$ is a delta then $\mathcal{C}'$ is a poset.
In particular, the second subdivision $\mathcal{C}''$ of a small category $\mathcal{C}$ is a poset.
\end{proposition}
The subdivision of a small category $\mathcal{C}$ induces a functor $\mathbf{d}:\mathcal{C}'\rightarrow\mathcal{C}$ defined on objects by $\mathbf{d}\tau=\tau(0)$ and on maps by
$\mathbf{d}[\tau, \sigma, v]=v$. Each functor $f:\mathcal{D}\rightarrow\mathcal{C}$ induces a functor $f':\mathcal{D}'\rightarrow\mathcal{C}'$ by taking
$f'(\xymatrix{[\mathbf{p}]\ar[r]^\tau & \mathcal{D}})$ to be the composite $\xymatrix{[\mathbf{p}]\ar[r]^\tau & \mathcal{D}\ar[r]^f & \mathcal{C}}$ and
$f'([\tau, \sigma, v])=[f'(\tau), f'(\sigma), f(v)]$. The subdivision is functor from the category of small categories to itself and $\mathbf{d}$ is a natural transformation from this functor to the identity functor.

In [2] M. Gerstenhaber and S. D. Schack used the functor $\mathbf{d}:\mathcal{C}'\rightarrow\mathcal{C}$ to ``subdivide" any diagram $\mathbb{A}$ over $\mathcal{C}$ obtaining a new diagram over $\mathcal{C}'$, $\mathbf{d}^{*}\mathbb{A}=\mathbb{A}'$, as follows:
 \begin{center} $(\mathbb{A}')^\tau=\mathbb{A}^{d\tau} $\;$ $\;$\mathrm{and} $\;$ $\;$ \varphi_{\mathbb{A}'}^{[\tau, \sigma, v]}=\varphi_{\mathbb{A}}^{d[\tau, \sigma, v]}=\varphi_\mathbb{A}^v$\end{center} It follows from the general case described in the previous section that the induced functor $\mathbf{d}^*:\mathbb{A}$-$\mathbf{mod}$$\rightarrow\mathbb{A}'$-$\mathbf{mod}$ preserves allowability and has a left adjoint $\mathbf{d}_!$ which preserves relative projectives. It is not hard to see that $\mathbf{d}^*$ is full and faithful, so we have $\mathbf{d}_!\mathbb{M}'\cong\mathbb{M}$. In fact,
  M. Gerstenhaber and S. D. Schack proved the more general result:
  \newpage
  \begin{theorem} \begin{bf} The Invariance Theorem \end{bf}

  The natural transformation induced by $\mathbf{d}^*$ induces and isomorphism $$\mathbf{Ext^\bullet_{\mathbb{A},k}}(-, -)\rightarrow\mathbf{Ext^\bullet_{\mathbb{A'},k}}((-)', (-)')$$ \end{theorem}

    Our effort in the next section is to generalize this theorem. We show that in a certain derived category context, where we may view the relative Yoneda cohomology as homomorphism groups, the extension of $\mathbf{d}^*$ is full and faithful. This result combined with our work in [5] and [6], on M. Gerstenhaber and S. D. Schack's $Special$ $Cohomology$ $Comparison$ $Theorem$,  constitutes a generalization of their $General$ $Cohomology$ $Comparison$ $Theorem$.

  A very important ingredient in our work is the following  theorem of M. Gerstenhaber and S. D. Schack.

 \begin{theorem} Let $\mathbb{N}$ be an $\mathbb{A}$-module. There is a relative projective allowable resolution $\mathbb{N}_\bullet\rightarrow\mathbb{N}'$ of the subdivided module $\mathbb{N}'$ in $\mathbb{A}'$-$\mathbf{mod}$ such that $\mathbf{d}_!\mathbb{P}_\bullet\rightarrow\mathbb{N}$ is a relative projective allowable resolution of $\mathbb{N}$ in $\mathbb{A}$-$\mathbf{mod}$.\end{theorem}

 They obtained the resolution $\mathbb{P}_\bullet\rightarrow\mathbb{N}'$  by ``spreading out" the $\mathbf{GSB}$ resolution of $\mathbb{N}$ over $\mathcal{C}'$. The resolution $\mathbf{d}_!\mathbb{P}_\bullet\rightarrow\mathbb{N}$ is exactly the $\mathbf{GSB}$ resolution. One might think that every $\mathbb{A}'$-module $\mathbb{N}$ has a relative projective allowable resolution $\mathbb{P}_\bullet\rightarrow\mathbb{N}$ for which $\mathbf{d}_!\mathbb{P}_\bullet\rightarrow\mathbf{d}_!\mathbb{N}$ is also a relative projective allowable resolution, but in general this is not true. When $k$ is a field this is equivalent to the exactness of $\mathbf{d}_!$, so it can't hold in general.

 \section{The Invariance Theorem}
 We construct now the relative derived category $\mathcal{D}^-_k(\mathbb{A}-\mathbf{mod})$ in which the Yoneda cohomology of $\mathbb{A}$-$\mathbf{mod}$ can be identified with the homomorphism groups. This construction is inspired by the ideas of [5] and [6],   where we defined the relative derived category of $\mathbb{A}$-$\mathbf{bimod}$. The word ``$\mathbf{relative}$"  is a
reminder to the reader that Yoneda cohomology is a relative theory,  since $k$ is a commutative ring that is not
necessarily a field.

 Let $\mathbb{A}$ be a diagram over $\mathcal{C}$ and let $Kom^{-}(\mathbb{A}-\mathbf{mod})$ be the category of bounded to
the right complexes of $\mathbb{A}$-modules
$$\mathbb{M_\bullet}:=\xymatrix{\cdots
\mathbb{M}_n\ar[r]&\cdots&
\cdots\ar[r]&\mathbb{M}_1\ar[r]& \mathbb{M}_0\ar[r]&0}$$

A map between two complexes $\mathbb{M_\bullet}$ and
$\mathbb{N_\bullet}$ is a collection of maps
$f=(f_i):\mathbb{M}_i\rightarrow\mathbb{N}_i$, one for each positive
integer $i$, which commute with the differentials of
$\mathbb{M_\bullet}$ and $\mathbb{N_\bullet}$. We do not require the
maps defining the complexes or the maps between complexes to be
$k$-split. We denote the homotopic category of
$Kom^{-}(\mathbb{A}-\mathbf{mod})$ by $\mathcal{K}^{-}(\mathbb{A}-\mathbf{mod})$.

\begin{definition}
A map
$\xymatrix{\mathbb{M_{\bullet}}\ar[r]^f&\mathbb{N}_{\bullet}}$ in
$Kom^{-}(\mathbb{A}-\mathbf{mod})$ is a $\mathbf{relative}$
$\mathbf{quasi}$-$\mathbf{isomorphism}$ if the maps of complexes of $\mathbb{A}^i$-modules
$\xymatrix{\mathbb{M}_{\bullet}^i\ar[r]^{f^i}&\mathbb{N}_{\bullet}^i}$
have  contractible cones, when considered as complexes of
$k$-modules, for all $i\in\mathcal{C}$.
\end{definition}

The following proposition characterizes relative quasi-isomorphisms and is an ingredient in proving that the class of these maps is localizing in $\mathcal{K}^{-}(\mathbb{A}-\mathbf{mod})$.
It was proved in [5] and [6] for $\mathbb{A}$-bimodules and it may be extended to $\mathbb{A}$-modules.
\begin{proposition}
Let $A$ be any $k$-algebra and $f:M_{\bullet}\longrightarrow
N_{\bullet}$ a map of complexes of $A$-modules in
$Kom^{-}(A-\mathbf{mod})$. Then, $f$ is a relative quasi-isomorphism if and
only if there exists $\gamma:N_{\bullet}\longrightarrow M_{\bullet}$
a map of complexes of $k$-modules such that $f\gamma\sim
id_{N_{\bullet}}$  and $\gamma f\sim id_{M_{\bullet}}$ in
$Kom^{-}(k-\mathbf{mod})$, where `$\sim$' stands for homotopy equivalence.
\end{proposition}
\begin{proof}
$'\Rightarrow'$

Let $f$ be a relative quasi-isomorphism. Since
$\mathcal{C}{(f)}_{\bullet }$ is contractible, when regarded as a
complex of $k$-modules,  there exist
$s=(s_n):\mathcal{C}{(f)}_{\bullet}^{n-1}\longrightarrow\mathcal{C}{(f)}_{\bullet}^{n}$
maps of $k$-modules such that
$sd_{\mathcal{C}{(f)}_{\bullet}}+d_{\mathcal{C}{(f)}_{\bullet}}s=id$.
We may assume that
\begin{center} $s=\left(%
\begin{array}{cc}
  \alpha & \gamma\\
  \beta& \delta \\
\end{array}%
\right)\mathrm{and}$ $d_{\mathcal{C}{(f)}_{\bullet}}=\left(%
\begin{array}{cc}
  -d_{M_\bullet} & 0\\
  f & d_{N_\bullet} \\
\end{array}%
\right),$\end{center} where $\alpha:M_{\bullet -1}\longrightarrow M_{\bullet}$,
$\beta:M_{\bullet -1}\longrightarrow N_{\bullet +1}$,
$\gamma:N_{\bullet}\longrightarrow M _{\bullet}$ and
$\delta:N_{\bullet}\longrightarrow N_{\bullet +1}$ are $k$ linear
maps. Since
$sd_{\mathcal{C}{(f)}_{\bullet}}+d_{\mathcal{C}{(f)}_{\bullet}}s=id$,
we obtain  $-\alpha d_{M_{\bullet}}+\gamma
f-d_{M_{\bullet}}\alpha=id_{M_{\bullet}}$, $-\beta
d_{M_{\bullet}}+\delta f+f\alpha+d_{N_{\bullet}}\beta=0$, $\delta
d_{N_\bullet}+f\gamma+d_{N_\bullet}\delta=id_{N_\bullet}$, and $\gamma
d_{N_\bullet}-d_{M_\bullet}\gamma=0 $.

The last relation implies that $\gamma$ is a map of complexes of $k$-modules. Since
$\delta
d_{N_{\bullet}}+d_{N_{\bullet}}\delta=id_{N_{\bullet}}-f\gamma$ and
$\alpha d_{M_{\bullet}}+d_{M_{\bullet}}\alpha=\gamma
f-id_{M_{\bullet}}$, we have that $f\gamma \sim id_{N_{\bullet}}$ and
$\gamma f\sim id_{M_{\bullet}}$ in $Kom^{-}{(k-\mathbf{mod})}$.

$'\Leftarrow'$

Let $f\gamma \sim id_{N_{\bullet}}$ and $\gamma f\sim
id_{M_{\bullet}}$ in $Kom^{-}{(k-\mathbf{mod})}$. This means that there are maps
$s^{N_{\bullet}}$ and $s^{M_{\bullet}}$ such that $f\gamma -
id_{N_{\bullet}}=s^{N_{\bullet}}d_{N_{\bullet}}+d_{N_{\bullet}}s^{N_{\bullet}}$
and $\gamma f-
id_{M_{\bullet}}=s^{M_{\bullet}}d_{M_{\bullet}}+d_{M_{\bullet}}s^{M_{\bullet}}$.

The map $s^{\mathcal{C}{(f)}}_{\bullet}=\left(%
\begin{array}{cc}
  s^{M_{\bullet}}+\gamma(s^{N_{\bullet}}f-fs^{M_{\bullet}}) & \gamma \\
  s^{N_{\bullet}}(fs^{M_{\bullet}}-s^{N_{\bullet}}f) & -s^{N_{\bullet}} \\
\end{array}%
\right) $ is a homotopy. \begin{center}
$s^{\mathcal{C}{(f)}}_{\bullet}d_{\mathcal{C}{(f)}_{\bullet}}+
d_{\mathcal{C}{(f)}_{\bullet}}s^{\mathcal{C}{(f)}}_{\bullet}=$
\end{center} \begin{center} $\left(%
\begin{array}{cc}
  id_{M_{\bullet}}-\gamma s^{N_{\bullet}}fd_{M_{\bullet}}+\gamma fs^{M_{\bullet}}d_{M_{\bullet}}-d_{M_{\bullet}}\gamma s^{N_{\bullet}}f+d_{M_{\bullet}}\gamma fs^{M_{\bullet}}& 0 \\
  s^{N_{\bullet}}s^{N_{\bullet}}fd_{M_{\bullet}}+f\gamma s^{N_{\bullet}}f- d_{N_{\bullet}}s^{N_{\bullet}}s^{N_{\bullet}}f-s^{N_{\bullet}}f\gamma f& id_{N_{\bullet}} \\
\end{array}%
\right)$=\end{center}

\begin{center}$\left(%
\begin{array}{cc}
   id_{M_{\bullet}} & 0 \\
  0 &  id_{N_{\bullet}} \\
\end{array}%
\right)= id_{\mathcal{C}{(f)}_{\bullet}}.$\end{center} This implies that
$\mathcal{C}{(f)}_{\bullet}$ is contractible in $Kom^{-}(k-\mathbf{mod}).$
\end{proof}

As a  corollary we note that if any two of $f, g$ or
$fg$ are relative quasi-isomorphisms then so is the third.
\begin{proposition} The class of relative quasi-isomorphisms in the homotopic
category $\mathcal{K}^{-}(\mathbb{A}-\mathbf{mod})$ is localizing.
\end{proposition}
\begin{proof}
By the previous proposition it remains to justify the extension conditions and the left-right equivalence condition.

That is, for every $ f\in
Mor_{\mathcal{K}^{-}(\mathbb{A}-\mathbf{mod})}$ and $s$ relative
quasi-isomorphism there exist
$g\in$$Mor_{\mathcal{K}^{-}(\mathbb{A}-\mathbf{mod})}$ and $t$
relative quasi-isomorphism such that the following squares

\begin{center}$\xymatrix{\mathbb{N}_{\bullet}\ar[r]^f
\ar[d]_{t}&\mathbb{M}_{\bullet}\ar[d]^
{s}\\
\mathbb{K}_{\bullet}\ar[r]^{g}& \mathbb{L}_{\bullet}}$ resp.
$\xymatrix{\mathbb{L}_{\bullet}\ar[r]^{g}\ar[d]_{s}&\mathbb{K}_{\bullet}\ar[d]^{t}\\
\mathbb{M}_{\bullet}\ar[r]^{f}& \mathbb{N}_{\bullet}}$\end{center}
are commutative (extension). In addition, given  $f,g$ two morphisms from $\mathbb{N}_\bullet$ to
$\mathbb{M}_\bullet$, the existence of a relative
quasi-isomorphism $s$ such that $sf=sg$ is equivalent to the existence of a
relative quasi-isomorphism $t$ such that $ft=gt$ (left-right equivalence).

In [1], chapter 3, theorem 4  states that the
class of quasi-isomorphisms (not relative) in the homotopic category
of an abelian category is localizing. The proof of  this theorem can be used entirely so we
will not reproduce it here.  To see that the extension requirement is true one should note that the cone
of the map $t$ constructed in [1] is the same, in
$\mathcal{K}^{-}(\mathbb{A}-\mathrm{mod})$, as the cone of $s$ and hence it is contractible. For the left-right equivalence one needs to note that the cone of the map $t$ constructed is the cone of $s$ shifted by 1, so it is contractible again.

\end{proof}

We define now the relative derived category of $\mathbb{A}$-$\mathbf{mod}$.

\begin{definition}
$\mathcal{D}_{k}^{-}(\mathbb{A}-\mathbf{mod}):=\mathcal{K}^{-}(\mathbb{A}-\mathbf{mod})(\Sigma^{-1}),$
where $\mathcal{K}^{-}$ is the homotopy category and $\Sigma$ is the class of relative quasi-isomorphisms in $\mathcal{K}^{-}(\mathbb{A}-\mathbf{mod})$.
\end{definition}
Because $\Sigma$ is localizing  we may regard the morphisms in $\mathcal{D}_{k}^{-}(\mathbb{A}-\mathbf{mod})$ as equivalence
classes of diagrams $$\xymatrix{&U\ar[dl]_s\ar[dr]^f\\
                    X&&Y}$$
The maps $s$ and $f$ are morphisms in the homotopy category with
$t\in\Sigma$. These diagrams are usually called roofs and we adopt
this terminology. In addition, because $\Sigma$ is a localizing class
the relative derived categories is triangulated.

\begin{proposition}
Let $\mathbb{P}_{\bullet}$ be a complex of relative projective
$\mathbb{A}$-modules and
$\xymatrix{\mathbb{M}_{\bullet}\ar[r]^f&\mathbb{N}_{\bullet}}$ a
relative quasi-isomorphism. Then

a) $Mor_{\mathcal{K}^{-}(\mathbb{A}-\mathbf{mod})}
(\mathbb{P}_{\bullet},\mathcal{C}{(f)}_{\bullet})=0$

b) The canonical map induced by $f$ $$\xymatrix{Mor_{\mathcal{K}^{-}(\mathbb{A}-\mathbf{mod})}(\mathbb{P}_\bullet, \mathbb{M}_\bullet)\ar[r]^{\mathbf{f}} & Mor_{\mathcal{K}^{-}(\mathbb{A}-\mathbf{mod})}(\mathbb{P}_\bullet,\mathbb{N}_\bullet)}$$ is onto.

c) The canonical map
$$\xymatrix{
Mor_{\mathcal{K}^{-}(\mathbb{A}-\mathbf{mod})}(\mathbb{P}_{\bullet},\mathbb{Q}_{\bullet})
\ar[r]^{\mathbf{can}}&Mor_{\mathcal{D}^{-}_{k}(\mathbb{A}\mathbf{-mod})}
(\mathbb{P}_{\bullet},\mathbb{Q}_{\bullet})}$$ is an isomorphism for
every $\mathbb{Q}_{\bullet}\in\ Kom^{-}(\mathbb{A}-\mathbf{mod})$.

\end{proposition}
\begin{proof}
a) Since $f$ is a relative quasi-isomorphism the cone
$\mathcal{C}{(f)}_{i}$ is acyclic and allowable $(\forall)
i\in\mathcal{C}$. If $g\in
Mor_{\mathcal{K}^{-}(\mathbb{A}-\mathbf{mod})}
(\mathbb{P}_{\bullet},\mathcal{C}{(f)}_{\bullet})$ then we prove that
$g=(g)_i:\mathbb{P}_i\longrightarrow\mathcal{C}(f)_i, i\geq 0$ is
homotopic to 0 inductively. Since $\mathbb{P}_0$ is a complex of
relative projective $\mathbb{A}$-modules we obtain that the map
$g_0$ from $\mathbb{P}_0$ to $\mathcal{C}{(f)}_0$ can be lifted to a
map $\delta_0:\mathbb{P}_0\longrightarrow \mathcal{C}(f)_1$ such
that $d_{\mathcal{C}(f)_1}\delta_0=g_0$. The image of
$g_1-\delta_0d_{\mathbb{P}_1}$ is contained in the image of
$d_{\mathcal{C}(f)_1}$ so it has a lifting
$\delta_1:\mathbb{P}_1\longrightarrow\mathcal{C}(f)_2$ such that
$d_{\mathcal{C}(f)_2}\delta_1=g_1-\delta_0d_{\mathbb{P}_1}$. Now,
the image of $g_2-\delta_1d_{\mathbb{P}_2}$ is contained in the
image of  $d_{\mathcal{C}(f)_2}$ and the conclusion follows
inductively.

b) The triangle
$\xymatrix{\mathbb{M}_\bullet\ar[r]^f & \mathbb{N}_\bullet\ar[r] & \mathcal{C}(f)_\bullet\ar[r] &\mathbb{M}_\bullet[1]}$ is induced by $f$.
Applying $Mor_{\mathcal{K}^{-}(\mathbb{A}-\mathbf{mod})}(\mathbb{P}_\bullet, (-))$ to it and using part a) we get the that the canonical map $\mathbf{f}$ is onto.

c) We show that $\mathbf{``can"}$ is injective. Assume that the roofs induced by the maps $\xymatrix{\mathbb{P}_\bullet\ar[r]^f & \mathbb{Q}_\bullet}$ and $\xymatrix{\mathbb{P}_\bullet\ar[r]^g & \mathbb{Q}_\bullet}$
are equivalent in $\mathcal{D}^{-}_k(\mathbb{A}-\mathbf{mod})$.
This implies that in  $\mathcal{K}^{-}(\mathbb{A}-\mathbf{mod})$ we have a commutative diagram
 $$
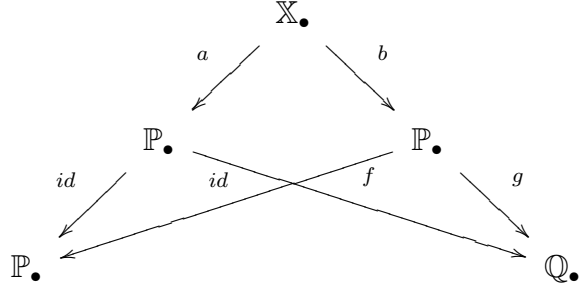
\xymatrix{&&\mathbb{X}_{\bullet}\ar[dl]_{a}\ar[dr]^{b}\\
 &\mathbb{P}_{\bullet}\ar[dl]_{id}\ar[drrr]^f&&\mathbb{P}_{\bullet}\ar[dlll]_{id}\ar[dr]^g\\
 \mathbb{P}_{\bullet}&&&&\mathbb{Q}_{\bullet}}$$ with $a$ and $b$ relative quasi-isomorphisms.
Thus $a=b$ and $f a=g b$.

Since $a$ is a relative quasi-isomorphism, part b) implies the existence of a map $l\in Mor_{\mathcal{K}^{-}(\mathbb{A}-\mathbf{mod})}(\mathbb{P_{\bullet}},\mathbb{X}_\bullet)$ such that
$al=id_{\mathbb{P}_\bullet}$. Now the injectivity
follows since $f=fal=gbl=g$.

To show that the map $\mathbf{``can"}$ is surjective we consider an arbitrary roof
$$
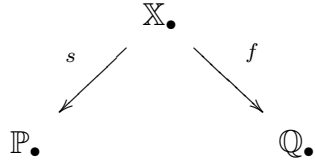
\xymatrix{&\mathbb{X}_{\bullet}\ar[dr]^f\ar[dl]_{s}\\
\mathbb{P}_{\bullet}&&\mathbb{Q}_{\bullet}}$$ in $Mor_{\mathcal{D}_k^{-}
({\mathbb{A}}-{\mathbf{mod}})}(\mathbb{P}_{\bullet},
\mathbb{Q}_{\bullet})$. Using part b) again there exist a map $t\in Mor_{\mathcal{K}^{-}(\mathbb{A}-\mathbf{mod})}(\mathbb{P}_\bullet, \mathbb{X}_\bullet)$
 such that $st=id_{\mathbb{P}_{\bullet}}$ in $\mathcal{K}^{-}(\mathbb{A}-\mathbf{mod})$. Since $s$ is a
relative quasi-isomorphism then so is $t$, so we have the commutative diagram \begin{center}
$
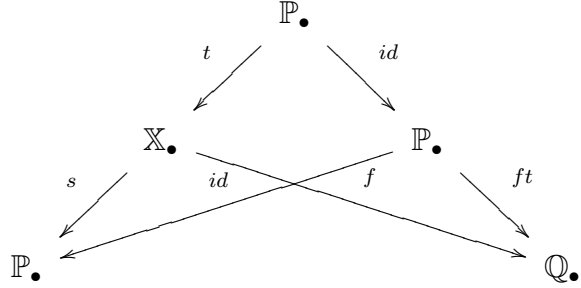
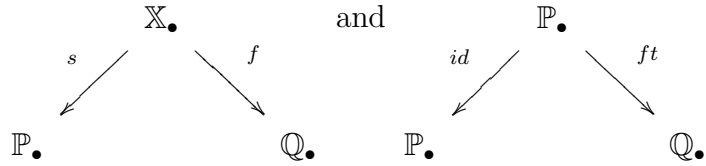
\xymatrix{&&\mathbb{P}_{\bullet}\ar[dl]_t\ar[dr]^{id}\\
                   &\mathbb{X}_{\bullet}\ar[dl]_s\ar[drrr]^f&&\mathbb{P}_{\bullet}
                   \ar[dlll]_{id}\ar[dr]^{ft}\\
                   \mathbb{P}_{\bullet}&&&&\mathbb{Q}_{\bullet}}$\end{center}
Therefore the roofs
$$\xymatrix{&\mathbb{X}_{\bullet}\ar[dl]_s\ar[dr]^f\\
            \mathbb{P}_{\bullet}&&\mathbb{Q}_{\bullet}}\mathrm{and}
            \xymatrix{&\mathbb{P}_{\bullet}\ar[dl]_{id}\ar[dr]^{f
            t}\\
            \mathbb{P}_{\bullet}&&\mathbb{Q}_{\bullet}}$$
are equivalent and since the second is the image of $ft$ through the canonical map it follows that $\mathbf{``can"}$ is surjective.\end{proof}
The proposition helps us establish the connection between the relative Yoneda cohomology of $\mathbb{A}-\mathbf{mod}$ and ${\mathcal{D}^{-}_{k}(\mathbb{A}-\mathbf{mod})}$.
\begin{theorem}
$\mathbf{Ext}^{i}_{\mathbb{A}, \mathbf{k}}(\mathbb{M},\mathbb{N})\simeq
Mor_{\mathcal{D}^{-}_{k}(\mathbb{A}-\mathbf{mod})}
(\mathbb{M}_{\bullet},\mathbb{N}_{\bullet}{[i]}).$
\end{theorem}
\begin{proof}

To see this, take the $\mathbf{GSB}$ resolution, $\mathcal{B}(\mathbb{M}_\bullet)$, of $\mathbb{M}$. Using the previous proposition
we get
$\mathbf{Ext}^{i}_{\mathbb{A}, \mathbf{k}}(\mathbb{M},\mathbb{N})=
H^{i}(Hom_{\mathbb{A}-\mathbf{mod}}(\mathcal{B}\mathbb{M}_{\bullet},\mathbb{N}))$\\$=
Mor_{\mathcal{K}^{-}(\mathbb{A}-\mathbf{mod})}(\mathcal{B}\mathbb{M}_{\bullet},\mathbb{N}_{\bullet}{[i]})
\cong
Mor_{\mathcal{D}^{-}_{k}(\mathbb{A}-\mathbf{mod})}(\mathcal{B}\mathbb{M}_{\bullet},\mathbb{N}_{\bullet}{[i]})\cong$\\$\cong
Mor_{\mathcal{D}^{-}_{k}(\mathbb{A}-\mathbf{mod})}(\mathbb{M}_{\bullet},\mathbb{N}_{\bullet}{[i]}).$
\end{proof}

The next result gives sufficient conditions for the total complex of a double complex to be homotopic equivalent with its augmented column.
\begin{proposition}

Let $A$ be a $k$-algebra and assume that we have a double complex of
$A$-modules
$$\xymatrix{
&\vdots\ar[d]^{d^1}&\vdots\ar[d]^{d^0}&\vdots\ar[d]^{d_M}\\
\cdots\ar[r]^{d_2}&X_{12}\ar[r]^{d_2}\ar[d]^{d^1}&X_{02}\ar[r]^{\varepsilon_2}\ar[d]^{d^0}&M_2\ar[d]^{d_M}\ar[r]&0\\
\cdots\ar[r]^{d_1}&X_{11}\ar[r]^{d_1}\ar[d]^{d^1}&X_{01}\ar[d]^{d^0}\ar[r]^{\varepsilon_1}&M_1\ar[d]^{d_M}\ar[r]&0\\
\cdots\ar[r]^{d_0}&X_{10}\ar[r]^{d_0}&X_{00}\ar[r]^{\varepsilon_0}&M_0\ar[r]&0}$$
such that:

a) Each row is $k$-contractible. ( $i.e.$ There exist $k$-module
maps

$\xymatrix{X_{(h-1)i}\ar[r]^{t_i^h}&X_{hi}}$  such that
$d_it_i^{h+1}+t_i^hd_i=id_{X_{hi}}$.)

b) The following diagrams are commutative:

$$\xymatrix{X_{hi}\ar[d]_{d^h}&X_{(h-1)i}\ar[d]^{d^{h-1}}\ar[l]_{t_i^h}\\
           X_{h(i-1)}&X_{(h-1)(i-1)}\ar[l]_{t_{i-1}^h}}
\xymatrix{X_{0i}\ar[d]_{d^0}&M_i\ar[d]^{d_M}\ar[l]_{t_i^0}\\
           X_{0(i-1)}&M_{i-1}\ar[l]_{t_{i-1}^0}}$$
           for all
           $h,i\geq 0$, Then

1.$\xymatrix{M_{\bullet}\ar[rr]^{t_{\bullet}^0}&&(TotX_{\bullet
\bullet})}$ and $\xymatrix{(TotX_{\bullet
\bullet})\ar[rr]^{\varepsilon_{\bullet}}&&M_{\bullet}}$ are maps of
complexes of $k$-modules, where $\varepsilon_i=0$ on $X_{jh},
j+h=i$ if $j>0$.

2. $\varepsilon_{\bullet}t_{\bullet}^0=id_{M_{\bullet}}$ and
$t_{\bullet}^0\varepsilon_{\bullet}\sim id_{TotX_{\bullet \bullet}}$
in $Kom^{-}(k-\mathbf{mod}),$ where $\sim $=homotopy equivalence.

\end{proposition}
\begin{proof}
1. The map $t_{\bullet}^0 $ is a map of complexes by b) and
$\varepsilon_{\bullet}$ is a map of complexes because
$d_M\varepsilon_{i+1}=d^0\varepsilon_i$ and $\varepsilon_id_i=0$.

2. The only thing to prove here is
$t_{\bullet}^0\varepsilon_{\bullet}\sim id_{TotX_{\bullet \bullet}}$
in $Kom^{-}(k-\mathbf{mod}).$ For $n\geq 0$ we define the map
$\xymatrix{(TotX_{\bullet \bullet})^n\ar[r]^{h^n}&(TotX_{\bullet
\bullet})^{n+1}}$ by $h^n:=(t_0^{n+1},t_1^{n},\ldots,t_n^1,0)$. It
is a simple exercise to check that $h^{\bullet}d_{TotX_{\bullet
\bullet}}+d_{TotX_{\bullet
\bullet}}h^{\bullet}=id-t_{\bullet}^0\varepsilon_{\bullet}$.
\end{proof}

The proposition is a key ingredient in justifying the next theorem.
\begin{theorem}
For each $\mathbb{M}_{\bullet}\in\mathcal{D}^{-}_{k}(\mathbb{A}-\mathbf{mod})$ there exist a complex of relative projective $\mathbb{A}'-$modules $\mathcal{U}\mathbb{M}_{\bullet}\in\mathcal{D}^{-}_{k}(\mathbb{A}'-\mathbf{mod})$  and a relative quasi-isomorphism $\varepsilon:\mathcal{U}\mathbb{M}_\bullet\rightarrow\mathbb{M}'_\bullet$ in $\mathcal{D}^{-}_{k}(\mathbb{A}'-\mathbf{mod})$.
\end{theorem}

\begin{proof}
Using theorem 2.3, for each term $\mathbb{M}'_i$ of the complex $\mathbb{M}_\bullet$, $i\geq 0$ we
obtain a double complex of relative projective $\mathbb{A}'$-modules with augmented column $\mathbb{M}'_{\bullet}$.  The rows of this complex are resolutions  obtained  by ``spreading out" the $\mathbf{GSB}$ resolution of $\mathbb{M}$ over $\mathcal{C}'$. In addition, each such row has a contracting homotopy and for each $\tau\in\mathcal{C}'$ we obtain a double complex of $(\mathbb{A}')^\tau$-modules which satisfies the conditions of the previous proposition.  Thus, by taking the total complex of the double complex with augmented column $\mathbb{M}'_\bullet$ we obtain a complex of relative projective $\mathbb{A}'$-modules,
$\mathcal{U}\mathbb{M}_{\bullet}$ and a relative
quasi-isomorphism
$\xymatrix{\mathcal{U}\mathbb{M}_{\bullet}\ar[r]^{\varepsilon}&\mathbb{M}'_{\bullet}}.$ Moreover, by applying the functor $\mathbf{d}_!$ to this relative quasi-isomorphism and using theorem 2.3 in combination with the fact that $\mathbf{d}^*$ is full and faithful, we get that $\mathbf{d}_!(\varepsilon)$ is a relative quasi-isomorphism from the total complex obtained by taking the $\mathbf{GSB}$ resolution of each   $\mathbb{M}_i$, to $\mathbb{M}_\bullet$.
\end{proof}
Since the functor $\mathbf{d}^*:\mathbb{A}-\mathbf{mod}\rightarrow \mathbb{A}'-\mathbf{mod}$  is exact and preserves allowable maps then it  preserves  relative quasi-isomorphisms as well. Thus, it induces a functor $\mathbf{d}^*$  at the level of relative derived categories.  We now prove the following generalization of the $Invariance$ $Theorem$.
\begin{theorem}
The functor $\mathbf{d}^*:\mathcal{D}_k^{-}(\mathbb{A}-\mathbf{mod})\rightarrow\mathcal{D}_k^{-}(\mathbb{A'}-\mathbf{mod})$ is full and faithful. That is,
$$\xymatrix{Mor_{\mathcal{D}_k^{-}(\mathbb{A}-\mathbf{mod})}(\mathbb{M}_\bullet,\mathbb{N}_\bullet)\ar[r]^{\mathbf{d}^*} & Mor_{\mathcal{D}_k^{-}(\mathbb{A'}-\mathbf{mod})}(\mathbb{M'}_\bullet,\mathbb{N'}_\bullet)}$$ is an isomorphism of sets for all $\mathbb{M}_\bullet$ and $\mathbb{N}_\bullet$ in $\mathcal{D}_k^{-}(\mathbb{A}-\mathbf{mod})$.
\end{theorem}

\begin{proof}

Let  $$\xymatrix{&\mathbb{X}_\bullet\ar[dl]_s\ar[dr]^f\\ \mathbb{M}'_\bullet && \mathbb{N}'_\bullet}$$ be a roof in $Mor_{\mathcal{D}_k^{-}(\mathbb{A'}-\mathbf{mod})}(\mathbb{M'}_\bullet,\mathbb{N'}_\bullet)$. Take $\mathcal{U}\mathbb{M}_\bullet$ and $\varepsilon$ as in the previous theorem. Since $\xymatrix{\mathbb{X}_\bullet\ar[r]^s & \mathbb{M}'_\bullet}$ is a relative quasi-isomorphism and $\mathcal{U}\mathbb{M}_\bullet$ is a complex of relative projective $\mathbb{A}'$-modules then proposition 3.5 implies that there exist $q\in Mor_{\mathcal{K}^{-}(\mathbb{A}-\mathbf{mod})}(\mathcal{U}\mathbb{M}_\bullet, \mathbb{X}_\bullet)$ such that $qs=\varepsilon$. Moreover, $q$ is a relative quasi-isomorphism because both $s$ and $\varepsilon$ are. We have now the equivalence of roofs \begin{center} $\xymatrix{&\mathbb{X}_\bullet\ar[dl]_s\ar[dr]^f\\ \mathbb{M}'_\bullet && \mathbb{N}'_\bullet}$ and $\xymatrix{&\mathcal{U}\mathbb{M}_\bullet\ar[dl]_\varepsilon\ar[dr]^{fq}\\ \mathbb{M}'_\bullet && \mathbb{N}'_\bullet}$\end{center} because of the following commutative diagram $$\xymatrix{ &&\mathcal{U}\mathbb{M}_\bullet\ar[dl]_q\ar[dr]^{{id}_{\mathcal{U}\mathbb{M}_\bullet}}\\ &\mathbb{X}_\bullet\ar[dl]_s\ar[drrr]^f &&\mathcal{U}\mathbb{M}_\bullet\ar[dlll]_\varepsilon\ar[dr]^{fq}\\ \mathbb{M}'_\bullet &&&& \mathbb{N}'_\bullet}$$
If $\varepsilon_{\mathbb{M}_\bullet}$ and $\varepsilon_{\mathbb{N}_\bullet}$ denote the maps of complexes induced by the counit of the adjunction
$\xymatrix{\mathbb{A}-\mathbf{mod}\ar@<1ex>[r]^(.46){\mathbf{d}^*}
&\mathbb{A'}-\mathbf{mod,}\ar@<1ex>[l]^{\mathbf{d}_!}}$ then note that they are isomorphisms since the functor $\mathbf{d}^*$ is full and faithful. In addition,
we have that $\mathbf{d}_!(\varepsilon)$ is a relative quasi-isomorphism in $\mathcal{D}_k^{-}(\mathbb{A}-\mathbf{mod})$, so the roof $$\xymatrix{&\mathbf{d}_!\mathcal{U}\mathbb{M}_\bullet\ar[dl]_{\varepsilon_{\mathbb{M}_\bullet}\mathbf{d}_!(\varepsilon)}\ar[dr]^{\varepsilon_{\mathbb{N}_\bullet}\mathbf{d}_!(fq)}\\ \mathbb{M}_\bullet && \mathbb{N}_\bullet}$$ exists in this category.

We show now that the image of this roof through $\mathbf{d}^*$ is equivalent to $$\xymatrix{&\mathcal{U}\mathbb{M}_\bullet\ar[dl]_\varepsilon\ar[dr]^{fq}\\ \mathbb{M}'_\bullet && \mathbb{N}'_\bullet}$$
To see this note that if $\eta$ is the unit of the adjunction, naturally extended to complexes, we have $\mathbf{d}^*(\varepsilon_{\mathbb{M}_\bullet})\eta_{\mathbf{d}^*({\mathbb{M}_\bullet})}=id_{\mathbf{d}^*({\mathbb{M}_\bullet})}$. In addition, the functoriality of $\eta$ implies that $(\mathbf{d}^*\mathbf{d}_!)(\varepsilon)\eta_{\mathcal{U}\mathbb{M}_\bullet}=\eta_{\mathbf{d}^*(\mathbb{M}_\bullet)}\varepsilon$ and
$(\mathbf{d}^*\mathbf{d}_!)(fq)\eta_{\mathcal{U}\mathbb{M}_\bullet}=\eta_{\mathbf{d}^*(\mathbb{M}_\bullet)}fq$. Thus, we have $\mathbf{d}^*(\varepsilon_{\mathbb{M}_\bullet})(\mathbf{d}^*\mathbf{d}_!)(\varepsilon)\eta_{\mathcal{U}{\mathbb{M}_\bullet}}=\varepsilon$.
and  $\mathbf{d}^*(\varepsilon_{\mathbb{N}_\bullet})(\mathbf{d}^*\mathbf{d}_!)(fq)\eta_{\mathcal{U}{\mathbb{M}_\bullet}}=fq$, so the following diagram is commutative and the surjectivity is proved.
$$\xymatrix{&&\mathcal{U}\mathbb{M}_\bullet\ar[dl]_{id}\ar[dr]^{\eta_{\mathcal{U}\mathbb{M}_\bullet}}\\
&\mathcal{U}\mathbb{M}_\bullet\ar[dl]_{\varepsilon}\ar[drrr]^{fq} && \mathbf{d}^*\mathbf{d}_!\mathcal{U}\mathbb{M}_\bullet\ar[dlll]_(.55){\mathbf{d}^*(\varepsilon_{\mathbb{M}_\bullet})(\mathbf{d}^*\mathbf{d}_!)(\varepsilon)}
\ar[dr]^{\mathbf{d}^*(\varepsilon_{\mathbb{N}_\bullet})(\mathbf{d}^*\mathbf{d}_!)(fq)}\\
\mathbb{M}_\bullet&&&&\mathbb{N}_\bullet}$$
To prove that $\mathbf{d}^*$ is injective assume that the roofs \begin{center} $\xymatrix{&\mathbb{S}'_\bullet\ar[dr]^{\mathbf{d}^*f}\ar[dl]_{\mathbf{d}^*s}\\ \mathbb{M}'_\bullet &&\mathbb{N}'_\bullet}$ and $\xymatrix{&\mathbb{T}'_\bullet\ar[dr]^{\mathbf{d}^*g}\ar[dl]_{\mathbf{d}^*t}\\ \mathbb{M}'_\bullet &&\mathbb{N}'_\bullet}$ \end{center} are equivalent in $\mathcal{D}_k^{-}(\mathbb{A}'-\mathbf{mod})$. Thus, there exist a commutative diagram
$$\xymatrix{&&\mathbb{X}_\bullet\ar[dl]_{u}\ar[dr]^{h}\\ &\mathbb{S}'_\bullet\ar[dl]_{\mathbf{d}^*s}\ar[drrr]^{\mathbf{d}^*f} && \mathbb{T}'_\bullet\ar[dlll]_{\mathbf{d}^*t}\ar[dr]^{\mathbf{d}^*g}\\\mathbb{M}'_\bullet &&&& \mathbb{N}'_\bullet}$$ where $u$ is a relative quasi-isomorphism.
We will try to replace $\mathbb{X}_\bullet$ with a "better" complex. For this let $\alpha:\mathcal{U}\mathbb{S}_\bullet\rightarrow\mathbb{S}'_\bullet$ be a relative quasi-isomorphism as in the previous theorem. Because $u$ is a relative quasi-isomorphism and $\mathcal{U}\mathbb{S}$ is a complex or relative projective $\mathbb{A}'$-modules, proposition 3.5 implies that
$Mor_{\mathcal{K}^{-}(\mathbb{A}'-\mathbf{mod})}
(\mathcal{U}\mathbb{S}_{\bullet},\mathcal{C}{(u)}_{\bullet})=0$, so there exist a map $\beta:\mathcal{U}\mathbb{S}_\bullet\rightarrow\mathbb{X}_\bullet$ such that $u\beta=\alpha$. Moreover, since $\alpha$ and $u$ are relative quasi-isomorphisms then so is $\beta$.
This implies the commutativity of the roof
$$\xymatrix{&&\mathcal{U}\mathbb{S}_\bullet\ar[dl]_{\alpha=u\beta}\ar[dr]^{h\beta}\\ &\mathbb{S}'_\bullet\ar[dl]_{\mathbf{d}^*s}\ar[drrr]^{\mathbf{d}^*f} && \mathbb{T}'_\bullet\ar[dlll]_{\mathbf{d}^*t}\ar[dr]^{\mathbf{d}^*g}\\\mathbb{M}'_\bullet &&&& \mathbb{N}'_\bullet}$$
Applying the functor $\mathbf{d}_!$ we get the commutative diagram $$\xymatrix{&&\mathbf{d}_!\mathcal{U}{\mathbb{S}_\bullet}\ar[dl]_{\mathbf{d}_!\alpha}\ar[dr]^{\mathbf{d}_!h\beta}\\&\mathbf{d}_!\mathbb{S}'_\bullet
\ar[dl]_{\mathbf{d}_!\mathbf{d}^*s}\ar[drrr]^{\mathbf{d}_!\mathbf{d}^*f}&&\mathbf{d}_!\mathbb{T}'_\bullet\ar[dlll]_{\mathbf{d}_!\mathbf{d}^*t}\ar[dr]^
{\mathbf{d}_!\mathbf{d}^*g}\\ \mathbf{d}_!\mathbb{M}'_\bullet &&&& \mathbf{d}_!\mathbb{N}'_\bullet}$$
Because $\mathbf{d}_!(\alpha)$, $\varepsilon_{\mathbb{M}_\bullet}$ and $\mathbf{d}_!\mathbf{d}^*(s)$ are relative quasi-isomorphisms in $\mathcal{D}_k^{-}(\mathbb{A}-\mathbf{mod})$ we get that the last diagram is an equivalence of roofs in $\mathcal{D}_k^{-}(\mathbb{A}-\mathbf{mod})$.
Since
$\varepsilon_{\mathbb{R}_\bullet}:\mathbf{d}_!\mathbb{R}'_\bullet\rightarrow\mathbb{R}_\bullet$ is an isomorphism for all $\mathbb{R}_\bullet\in\mathcal{D}_k^{-}(\mathbb{A}-\mathbf{mod})$ we get that the roofs  \begin{center} $\xymatrix{&\mathbb{S}_\bullet\ar[dr]^{f}\ar[dl]_{s}\\ \mathbb{M}_\bullet &&\mathbb{N}_\bullet}$ and $\xymatrix{&\mathbb{T}_\bullet\ar[dr]^{g}\ar[dl]_{t}\\ \mathbb{M}_\bullet &&\mathbb{N}_\bullet}$ \end{center} are
equivalent in $\mathcal{D}_k^{-}(\mathbb{A}-\mathbf{mod})$, so injectivity is proved.
\end{proof}
We obtain as a corollary M. Gerstenhaber and S. D. Schack's $Invariance$ $ Theorem$.
\begin{corollary}
\begin{bf} The Invariance Theorem\end{bf}

The natural transformation induced by $\mathbf{d}^*$ induces and isomorphism $$\mathbf{Ext^\bullet_{\mathbb{A},k}}(-, -)\longrightarrow\mathbf{Ext^\bullet_{\mathbb{A'},k}}((-)', (-)')$$
\end{corollary}

\begin{proof} Theorem 3.6 implies that we have the isomorphisms \begin{center}$ \mathbf{Ext}^{i}_{\mathbb{A},\mathbf{k}}(\mathbb{M},\mathbb{N})\simeq
Mor_{\mathcal{D}^{-}_{k}(\mathbb{A}-\mathbf{mod})}
(\mathbb{M}_{\bullet},\mathbb{N}_{\bullet}{[i]})$  $\mathbf{Ext}^{i}_{\mathbb{A'},\mathbf{k}}(\mathbb{M}',\mathbb{N}')\simeq
Mor_{\mathcal{D}^{-}_{k}(\mathbb{A}'-\mathbf{mod})}
(\mathbb{M}'_{\bullet},\mathbb{N}'_{\bullet}{[i]})$\end{center}
Since $\mathbf{d}^*$ is full and faithful we get the desired isomorphism.

\end{proof}

\section{The General Cohomology Comparison Theorem}

To each diagram of algebras $\mathbb{A}$ over a $\mathbf{poset}$ $\mathcal{C}$, M. Gerstenhaber and S. D. Schack
associated a single algebra $\mathbb{A}!=$ of the row-finite
$\mathcal{C}\times\mathcal{C}$ matrices $(a_{ij})$ with
$a_{ij}\in\mathbb{A}^i$ if $i\leq j$ and $a_{ij}=0$ otherwise. The
addition is componentwise and the multiplication
$(a_{ij})(b_{ij})=(c_{ij})$ is induced by the matrix multiplication
with the understanding that, for $h\leq i\leq j$, the summand
$a_{hi}b_{ij}$ of $c_{hj}$ is regarded as
$a_{hi}b_{ij}=a_{hi}\varphi^{hi}(b_{ij})$. For our purpose it is
convenient to use the equivalent representation
$\mathbb{A!}=\prod_{i\in\mathcal{C}}\coprod_{i\leq
j}\mathbb{A}^i\varphi^{ij}$, as $k$-bimodule. Here $\varphi^{ij}$
serve to distinguish distinct copies of $\mathbb{A}^i$ from one
another. The general element of $\mathbb{A}^i\varphi^{ij}$ will be
denoted $a^{i}\varphi^{ij}$. The multiplication is defined
componentwise and subject to the rule:
$(a^h\varphi^{hi})(a^j\varphi^{jl})=a^h\varphi^{hi}(a^j)\varphi^{hl}$
if $i=j$ and $0$ otherwise.

Let $1_i$ the unit element of $\mathbb{A}^i$. Since
$(a^h\varphi^{hi})(1_i\varphi^{ij})=a^h\varphi^{hj}$ and
$(1_i\varphi^{hi})(a^i\varphi^{ij})=\varphi^{hi}(a^i)\varphi^{hj}$
we may abbreviate $1_i\varphi^{ij} $ to $\varphi^{ij}$. The maps
$\varphi^{ij}$ are then elements of $\mathbb{A}!$ and
$\varphi^{hi}\varphi^{ij}=\varphi^{hj}$;
$\varphi^{hi}\varphi^{jl}=0$ if $i\neq j$.

M. Gerstenhaber and S. D. Schack defined the functor
\begin{center}$!:\mathbb{A}^e-\mathbf{mod}\longrightarrow(\mathbb{A}!)^e-\mathbf{mod}$, such
that $\mathbb{A}\longrightarrow\mathbb{A}!$\end{center} by setting for any
$\mathbb{A}^e$-module $\mathbb{M}$,
\begin{center}$\mathbb{M}!=\prod_{i\in\mathcal{C}}\coprod_{i\leq
j}\mathbb{M}^i\varphi^{ij}$ as a $k$-bimodule.\end{center} The actions of
$\mathbb{A}!$ are defined by:
$$(a^h\varphi^{hi})(m^i\varphi^{ij})=a^hT_{\mathbb{M}}^{hi}(m^i)\varphi^{hj}$$
$$(m^h\varphi^{hi})(a^i\varphi^{ij})=m^h\varphi^{hi}(a^i)\varphi^{hj}$$
$$(a^h\varphi^{hi})(m^j\varphi^{jl})=0=(m^h\varphi^{hi})(a^j\varphi^{jl}),
\mathrm{if} i\neq j.$$ For $\eta\in
Hom_{\mathbb{A}^e}(\mathbb{N}, \mathbb{M})$ define
$\eta!\in Hom_{(\mathbb{A}!)^e}(\mathbb{N}!, \mathbb{M}!)$
by $\eta!(n^i\varphi^{ij})=\eta^i(n^i)\varphi^{ij}$.

The functor ! induces a functor between the relative derived categories $\mathcal{D}_k^{-}(\mathbb{A}^e-\mathbf{mod})$ and  $\mathcal{D}_k^{-}((\mathbb{A}!)^e-\mathbf{mod})$. We proved the following theorem about the induced functor.  (see [5] or [6])
\begin{theorem}
The functor $\xymatrix{\mathcal{D}_k^{-}
({\mathbb{A}^e}-{\mathbf{mod}})\ar[r]^(.45){!}&\mathcal{D}_k^{-}
({(\mathbb{A!})^e}-{\mathbf{mod}})}$ is full and faithful. That is,
$$\xymatrix{ Mor_{\mathcal{D}_k^{-}
({\mathbb{A}^e}-{\mathbf{mod}})}(\mathbb{M}_{\bullet},
\mathbb{N}_{\bullet})\ar[r]^(.48){!}&Mor_{\mathcal{D}_k^{-}
({(\mathbb{A!})^e}-{\mathbf{mod}})}(\mathbb{M}_{\bullet}!,
\mathbb{N}_{\bullet}!)}$$ is an isomorphism of sets for all
$\mathbb{M}_\bullet,
\mathbb{N}_\bullet\in\mathcal{D}_{k}^{-}(\mathbb{A}^e-\mathbf{mod})$.
\end{theorem}

As a corollary we obtained the \begin{it} Special Cohomology Comparison Theorem \end{it} due to M. Gerstenhaber and S. D. Schack (see [3]).
\begin{corollary}{(Special Cohomology Comparison Theorem)}\\
The functor ! induces an isomorphism of relative Yoneda cohomologies
$$Ext^{\bullet}_{\mathbb{A}^e, \mathbf{k}}((-),(-))\cong
Ext^{\bullet}_{(\mathbb{A}!)^e, \mathbf{k}}((-)!,(-)!).$$ In particular,
we have an isomorphism of relative Hochschild cohomologies
$$\mathbf{H}^{\bullet}(\mathbb{A},(-))\cong \mathbf{H}^{\bullet}(\mathbb{A}!,(-)!).$$
\end{corollary}

As we noted in proposition 2.1, the second subdivision of a small category is always a poset. Therefore, by combining theorems 3.9 and 4.1, we get the main result of this paper in the form of the following theorem.

\begin{theorem} Let $\mathcal{C}$ be an arbitrary small category and $\mathbb{A}$ be a diagram over $\mathcal{C}$. Then, the functor
$$\mathbb{M}_\bullet\longrightarrow(\mathbb{M}''_\bullet)!$$ between the categories $\mathcal{D}_k^{-}(\mathbb{A}^e-\mathbf{mod})$ and
$\mathcal{D}_k^{-}((\mathbb{A}''!)^e-\mathbf{mod})$ is full and faithful.
That is, the natural map $$\xymatrix{ Mor_{\mathcal{D}_k^{-}
({\mathbb{A}^e}-{\mathbf{mod}})}(\mathbb{M}_{\bullet},
\mathbb{N}_{\bullet})\ar[r]^(.45){(-)''!}&Mor_{\mathcal{D}_k^{-}
({(\mathbb{A}''!)^e}-{\mathbf{mod}})}((\mathbb{M})''_{\bullet}!,
(\mathbb{N})''_{\bullet}!)}$$ is an isomorphism of sets for all
$\mathbb{M}_\bullet,
\mathbb{N}_\bullet\in\mathcal{D}_{k}^{-}(\mathbb{A}^e-\mathbf{mod})$
\end{theorem}

As a corollary we obtain the \begin{it} General Cohomology Comparison Theorem\end{it}  of M. Gerstenhaber and S. D. Schack.

\begin{corollary} \begin{bf}  General Cohomology Comparison Theorem\end{bf}

Let $\mathcal{C}$ be a small category and $\mathbb{A}$ be a diagram over $\mathcal{C}$. Then the functor $$\mathbb{M}_\bullet\longrightarrow(\mathbb{M}''_\bullet)!$$ between the categories $\mathbb{A}^e-\mathbf{mod}$ and $(\mathbb{A}'')!^e-\mathbf{mod}$ is full and faithful. The induced map $$\xymatrix{\mathbf{Ext}^\bullet_{\mathbb{A}^e, \mathbf{k}}((-), (-))\ar[r]&\mathbf{Ext}^\bullet_{(\mathbb{A}'')!^e, \mathbf{k}, }((-)''!, (-)''!)},$$ $[\mathcal{E}]\longrightarrow[\mathcal{E''!}]$ is an isomorphism. In particular, there is an isomorphism of  Hochschild cohomologies
$$\mathbf{H}^{\bullet}(\mathbb{A},(-))\cong \mathbf{H}^{\bullet}((\mathbb{A}'')!,(-)''!).$$
\end{corollary}

\end{document}